\documentclass[12pt,oneside]{amsart}

\pdfoutput=1

\usepackage[section]{placeins}
\usepackage[margin=1in]{geometry}
\usepackage{xcolor}
\usepackage{graphicx}
\usepackage{amssymb}
\usepackage{latexsym}
\usepackage{enumitem}
\usepackage{multicol}
\usepackage{amsmath}
\usepackage{amsthm}
\usepackage{amscd}
\usepackage{multirow}
\usepackage{rotating}
\usepackage{float}
\usepackage{dsfont}
\usepackage{booktabs}

\usepackage[colorlinks,citecolor=red,pagebackref,hypertexnames=false]{hyperref}
\hypersetup{
  colorlinks = true,
  citecolor=red,
  linkcolor  = red
}

\setlist{itemsep=.06125in}

\restylefloat{table}
\numberwithin{equation}{section}

\theoremstyle{plain}
\newtheorem{theorem}{Theorem}[section]

\newtheorem{corollary}[theorem]{Corollary}

\theoremstyle{definition}
\newtheorem{definition}[theorem]{Definition}

\newtheorem{problem}[theorem]{Problem}

\newtheorem{assumption}[theorem]{Assumption}

\theoremstyle{remark}
\newtheorem{remark}[theorem]{Remark}

\def\FR{\operatorname{FR}}
\def\VC{\operatorname{VC}}

\date{\today}

\author{W. Burstein, A. Iosevich, and A. Sant}
\address{Department of Mathematics, University of Rochester, Rochester, NY, USA}
\email{willburst88@gmail.com}
\address{Department of Mathematics, University of Rochester, Rochester, NY, USA}
\email{iosevich@gmail.com}
\address{Department of Mathematics, University of Rochester, Rochester, NY, USA}
\email{asant2@ur.rochester.edu} 

\thanks{A. I. was supported in part by the National Science Foundation under NSF DMS - 2154232.}

\title{Arithmetic functions and learning theory}

\begin{document}

\begin{abstract}
We establish a connection between analytic number theory and computational learning theory by showing that the Möbius function belongs to a class of functions that is statistically hard to learn from random samples. Let $\mu_R$ denote the restriction of the Möbius function to the squarefree integers in $\{1,\dots,R\}$. Using a recent lower bound of Pandey and Radziwi{\l}{\l} for the $L^1$ norm of exponential sums with Möbius coefficients, we prove that
\[
\FR(\mu_R) \gg R^{-1/4-\epsilon}
\]
for every $\epsilon>0$. We then show that, for a suitable absolute constant $c_0>0$, the class of $\{-1,1\}$-valued functions on the squarefree integers with Fourier Ratio at least $c_0$ has Vapnik--Chervonenkis dimension at least $cR$. It follows that any distribution-independent learning algorithm that succeeds uniformly on the class $\mathcal{H}_R(\eta_R)$ containing $\mu_R$, where $\eta_R \to 0$, requires at least $\Omega(R)$ samples. We also discuss a conditional improvement under a strong uniform bound for additive twists of the Möbius function, and we note that the same method applies to the Liouville function.
\end{abstract}



\maketitle

\section{Introduction}

Let $\mu(n)$ denote the classical M\"obius function. Recall that $\mu(1)=1$, $\mu(n)=0$ if $n$ is not squarefree, and $\mu(n)=(-1)^k$ if $n$ is squarefree, where $k$ is the number of prime factors of $n$. The purpose of this paper is to show that the M\"obius function is difficult to learn in the distribution-independent setting, in the sense that any algorithm that succeeds uniformly over the hypothesis class $\mathcal{H}_R(\eta_R)$ containing $\mu_R$ requires a large number of samples.

\vskip.125in

We show that any distribution-independent learning algorithm that succeeds uniformly over the class $\mathcal{H}_R(\eta_R)$ containing $\mu_R$ and achieves nontrivial accuracy must use at least $\Omega(R)$ samples. The key analytic input is a lower bound for the Fourier Ratio of $\mu_R$, which serves as a complexity parameter that, through an appropriate thresholded hypothesis class, yields large Vapnik--Chervonenkis dimension.

The main idea is to use the Fourier Ratio, defined as the ratio of the $L^1$ and $L^2$ norms of the Fourier transform (see Definition~\ref{def:FR}), as a measure of complexity. A lower bound on the Fourier Ratio bounded away from zero leads to large combinatorial complexity, in the sense that the associated hypothesis class has large Vapnik--Chervonenkis (VC) dimension and therefore requires many samples to learn.

The Fourier Ratio has recently emerged as a useful tool in a variety of settings, including signal complexity, spectral synthesis, and approximation theory \cite{Aldahleh2025, ILPY25, IMW2026, DI26}.

There is a long-standing heuristic, supported by both conjectures and partial results, that the M\"obius function behaves in many respects like a random sequence. 
The Chowla conjecture predicts that the values of $\mu(n)$ exhibit asymptotic independence \cite{Tao2015}, while Sarnak's M\"obius randomness principle \cite{Sarnak2010} asserts that $\mu(n)$ is asymptotically orthogonal to all deterministic sequences. 
Even classical estimates such as $\sum_{n \le x} \mu(n) = o(x)$, and its refinement under the Riemann Hypothesis \cite{Davenport1980}, reflect cancellation consistent with random behavior. 
These perspectives suggest that the M\"obius function should be difficult to predict from partial information.

The perspective taken in this paper is different from the traditional probabilistic or dynamical viewpoints. 
Instead of studying correlations or orthogonality properties, we measure the complexity of $\mu$ through its Fourier transform and relate this to learnability in a precise computational sense. 
This provides a new bridge between analytic number theory and statistical learning theory, suggesting that classical problems about multiplicative functions can be reinterpreted in terms of complexity and sampling.

It is useful to place our results in the broader context of bounds for exponential sums with multiplicative coefficients. A classical consequence of Littlewood-type inequalities \cite{Konyagin1981,MPS1981} is that trigonometric polynomials with coefficients of modulus one have $L^1$ norm bounded below by a logarithmic factor. In particular, for multiplicative functions taking values in $\{-1,1\}$, one expects $\|F_R\|_{L^1} \gg \log R$. Recent and forthcoming work of Klurman and Bezin (preprint, in preparation) shows that, among multiplicative functions, the condition $\|F_R\|_{L^1} \ll \log R$ essentially characterizes Dirichlet characters. In contrast, the M\"obius and Liouville functions lie in a different regime, where the $L^1$ norm grows much faster (at least $R^{1/4-\epsilon}$ unconditionally), reflecting a higher level of arithmetic complexity.

Recent work of Pandey and Radziwi{\l}{\l} \cite{PandeyRadziwillII} provides a complementary structural perspective on exponential sums with multiplicative coefficients. They show that if a bounded multiplicative function $f$ satisfies 
\[
\int_0^1 \left| \sum_{n\le N} f(n)e(n\alpha)\right| d\alpha \le \Delta,
\]
then $f$ must closely resemble a Dirichlet character on primes. In other words, small $L^1$ norm forces strong arithmetic structure. From the viewpoint of the present paper, this result may be interpreted as showing that low Fourier complexity is only possible for highly structured multiplicative functions. In contrast, the M\"obius and Liouville functions exhibit large $L^1$ norms, placing them firmly in the high-complexity regime that leads to large VC dimension and statistical hardness of learning.

\vskip.125in 

\subsection{Main results}

Our first main analytic result is an unconditional lower bound for the Fourier Ratio of the M\"obius function, which serves as the key input for the learning-theoretic result stated above.

\begin{definition} \label{def:FR} 
Let
\[
X_R = \{1,2,\dots,R\}, \qquad X_R' = \{ n \le R : \mu(n) \neq 0 \}.
\]
Let $\mu_R : X_R' \to \{-1,1\}$ denote the restriction of the M\"obius function to $X_R'$, extended to $X_R$ by setting $\mu_R(n)=0$ for $n \notin X_R'$.
Define the Fourier transform
\[
\widehat{\mu_R}(x) = \sum_{n=1}^R \mu_R(n) e^{-2\pi i n x}, \qquad x \in [0,1].
\]
The Fourier Ratio is defined as
\[
\FR(\mu_R) = \frac{\|\widehat{\mu_R}\|_{L^1([0,1])}}{\|\widehat{\mu_R}\|_{L^2([0,1])}}.
\]

Similarly, for any function $f$ in our hypothesis class, 
$$ \widehat{f}(x)=\sum_{n=1}^R e^{-2 \pi i nx} f(n)$$ since the torus is the dual group of ${\mathbb Z}$. Similarly, 
$$ FR(f)=\frac{{||\widehat{f}||}_{L^1([0,1])}}{{||\widehat{f}||}_{L^2([0,1])}}.$$
\end{definition}

\vskip.125in 

\begin{theorem}[Unconditional]\label{thm:unconditional}
For every $\epsilon > 0$ and all sufficiently large $R$,
\[
\FR(\mu_R) \gg_\epsilon R^{-1/4-\epsilon}.
\]
\end{theorem}

The proof combines the classical asymptotic for the $L^2$ norm (which counts squarefree numbers) with a recent breakthrough of Pandey and Radziwi{\l}{\l} \cite{PandeyRadziwill} giving a lower bound for the $L^1$ norm of exponential sums with M\"obius coefficients.

We also discuss a conditional improvement. 
Let us state the following assumption, which is a strong uniform bound for additive twists of the M\"obius function.

\begin{assumption}\label{ass:strong_grh}
For every $\epsilon > 0$,
\[
\sup_{x \in [0,1]} \left| \sum_{n=1}^R \mu(n) e^{2\pi i n x} \right| \ll_\epsilon R^{1/2 + \epsilon}.
\]
\end{assumption}

\begin{remark}
Assumption \ref{ass:strong_grh} is not known, even under the Generalized Riemann Hypothesis. 
The best known uniform bound, due to Baker and Harman \cite{BakerHarman1991}, is $\ll_\epsilon R^{3/4 + \epsilon}$, and improving this exponent to $1/2$ remains a major open problem (see \cite{PandeyRadziwill} for discussion).
\end{remark}

\begin{theorem}[Conditional]\label{thm:conditional}
Under Assumption \ref{ass:strong_grh},
\[
\FR(\mu_R) \ge R^{-o(1)}.
\]
\end{theorem}

Theorems \ref{thm:unconditional} and \ref{thm:conditional} show that $\mu_R$ belongs to a hypothesis class $\mathcal{H}_R(\eta_R)$ with $\eta_R \to 0$, and this is sufficient to imply large VC dimension via the monotonicity of VC dimension.

\begin{theorem}[VC dimension lower bound] \label{thm:vc_lower}
There exists an absolute constant $c_0 > 0$ such that for all sufficiently large $R$, the class
\[
\mathcal{H}_R(c_0) = \{ f : X_R' \to \{-1,1\} : \FR(f) \ge c_0 \}
\]
satisfies $\VC(\mathcal{H}_R(c_0)) \ge c R$ for some absolute constant $c > 0$.
\end{theorem}

Combining these results with the classical VC lower bound from learning theory yields the following sample complexity lower bound.

\begin{theorem}[Main result-unconditional]\label{thm:main}
Let $\mu_R$ be as above. 
Any distribution-independent learning algorithm that succeeds uniformly over the class $\mathcal{H}_R(\eta_R)$ and receives $m$ random samples $(n_i, h^*(n_i))$ with $n_i$ drawn from any distribution on $X_R'$, and outputs a (possibly random) hypothesis $\hat{h}$ satisfying
\[
\sup_{h^* \in \mathcal{H}_R(\eta_R)} \mathbb{P}\left( \operatorname{error}(\hat{h}) \ge \frac{1}{4} \right) \le \frac{1}{3}
\]
must have
\[
m \ge c R
\]
for some absolute constant $c > 0$.
\end{theorem}

\begin{remark}
The conclusion of Theorem \ref{thm:main} is a uniform statement over the class $\mathcal{H}_R(\eta_R)$, rather than a statement about learning $\mu_R$ in isolation. This is essential for applying the Vapnik--Chervonenkis lower bound. Since $\mu_R \in \mathcal{H}_R(\eta_R)$, the theorem implies that any distribution-independent algorithm that guarantees uniform success over a class containing $\mu_R$ must use at least $\Omega(R)$ samples.
\end{remark}

\begin{remark}
There is an instructive contrast between different notions of complexity illustrated by the M\"obius function. From the viewpoint of Kolmogorov complexity, the function $\mu(n)$ is highly structured: it admits a very short description in terms of the prime factorization of $n$, and is therefore algorithmically simple. 

However, the results of this paper show that from the perspective of statistical learning, $\mu$ exhibits high complexity. More precisely, the large Fourier Ratio of $\mu_R$ implies that any distribution-independent algorithm that succeeds uniformly on a hypothesis class containing $\mu_R$ requires at least $\Omega(R)$ samples to achieve nontrivial accuracy. Thus, despite its low descriptional complexity, the M\"obius function belongs to hypothesis classes with high sampling complexity.

This highlights a fundamental distinction between algorithmic complexity and statistical complexity: a function may be easy to describe but hard to learn from data.
\end{remark}

\begin{remark}\label{rem:general_principle}
The proof of Theorem \ref{thm:unconditional} relies on two properties of the M\"obius function:
\begin{enumerate}
\item $\displaystyle \sum_{n=1}^R |\mu(n)|^2 \asymp R$ (i.e., $\mu(n)$ is non-zero on a positive proportion of integers);
\item $\displaystyle \int_0^1 \left| \sum_{n=1}^R \mu(n) e^{2\pi i n x} \right| dx \gg_\epsilon R^{1/4-\epsilon}$ (Pandey--Radziwill).
\end{enumerate}
The same properties hold for the Liouville function $\lambda(n) = (-1)^{\Omega(n)}$, where $\Omega(n)$ is the total number of prime factors counted with multiplicity. 
Indeed, Pandey and Radziwi{\l}{\l} prove their $L^1$ lower bound for $\lambda$ as well, and $\sum_{n=1}^R |\lambda(n)|^2 = R$. 
Hence the same unconditional lower bound $\FR(\lambda_R) \gg R^{-1/4-\epsilon}$ holds for the Liouville function, and consequently the same sample complexity lower bound applies to any distribution-independent algorithm that succeeds uniformly on a hypothesis class $\mathcal{H}_R(\eta_R)$ containing $\lambda_R$. Related perspectives on multiplicative functions and their unpredictability can be found in the work of Bourgain \cite{Bourgain2013} and Green \cite{Green2012}.
\end{remark}

\subsection*{Organization of the paper}

In Section \ref{sec:FR} we prove the unconditional Fourier Ratio lower bound (Theorem \ref{thm:unconditional}) and discuss the conditional improvement (Theorem \ref{thm:conditional}). 
In Section \ref{sec:VC} we prove the VC dimension lower bound (Theorem \ref{thm:vc_lower}). 
In Section \ref{sec:learning} we combine these results to obtain the sample complexity lower bound (Theorem \ref{thm:main}). 
Section \ref{sec:numerics} presents numerical experiments illustrating these results. 
Finally, Section \ref{sec:problems} discusses open problems and connections with the Riemann Hypothesis.

\section{Fourier Ratio for the M\"obius Function}
\label{sec:FR}

For $R \ge 1$, define the exponential sum
\[
F_R(x) = \sum_{n=1}^R \mu(n) e^{2\pi i n x}.
\]
We are interested in the Fourier Ratio
\[
\FR(\mu_R) = \frac{\|F_R\|_{L^1([0,1])}}{\|F_R\|_{L^2([0,1])}}.
\]

\subsection{The $L^2$ norm}

By Parseval's identity on the circle,
\[
\|F_R\|_{L^2}^2 = \sum_{n=1}^R |\mu(n)|^2 = \sum_{n=1}^R \mu(n)^2.
\]
The sum $\sum_{n=1}^R \mu(n)^2$ counts the number of squarefree integers up to $R$. 
A classical result of analytic number theory states that
\[
\sum_{n=1}^R \mu(n)^2 = \frac{6}{\pi^2} R + O(\sqrt{R}),
\]
which follows from the identity $\sum_{n=1}^\infty \mu(n)^2 n^{-s} = \zeta(s)/\zeta(2s)$ and standard Tauberian arguments \cite{Apostle1976}. 
Thus
\[
\|F_R\|_{L^2}^2 = \frac{6}{\pi^2} R + O(\sqrt{R}),
\]
and therefore
\[
\|F_R\|_{L^2} = \sqrt{\frac{6}{\pi^2}} \, R^{1/2} + O(1).
\]

\subsection{Unconditional $L^1$ lower bound}

The key arithmetic input is the following theorem of Pandey and Radziwi{\l}{\l}.

\begin{theorem}[Pandey--Radziwill \cite{PandeyRadziwill}, Theorem 1]\label{thm:pr}
For every $\epsilon > 0$,
\[
\int_0^1 \left| \sum_{n \le X} \mu(n) e^{2\pi i n \alpha} \right| d\alpha \gg_\epsilon X^{1/4 - \epsilon}.
\]
The same bound holds for the Liouville function $\lambda(n)$.
\end{theorem}

\begin{remark}
The proof in \cite{PandeyRadziwill} is presented for the Liouville function, but the authors note that the same method works for the M\"obius function, improving the earlier bound $\gg X^{1/6}$ due to Balog and Ruzsa \cite{BalogRuzsa2001}.
\end{remark}

Theorem \ref{thm:pr} directly gives the desired lower bound for $F_R$ (with $X=R$), so
\[
\|F_R\|_{L^1} \gg_\epsilon R^{1/4 - \epsilon}.
\]

\begin{proof}[Proof of Theorem \ref{thm:unconditional}]
Combining the $L^2$ estimate and the Pandey--Radziwill $L^1$ lower bound,
\[
\FR(\mu_R) = \frac{\|F_R\|_{L^1}}{\|F_R\|_{L^2}} \gg_\epsilon \frac{R^{1/4 - \epsilon}}{R^{1/2}} = R^{-1/4 - \epsilon}.
\]
This holds for every $\epsilon > 0$ and all sufficiently large $R$. 
\end{proof}

\subsection{Conditional improvement}

We now discuss what could be proved under a strong assumption on the $L^\infty$ norm of $F_R$.

\begin{assumption}\label{ass:l_infty}
Assumption~\ref{ass:strong_grh} holds.
\end{assumption}

\begin{remark}
Assumption \ref{ass:l_infty} is a strong form of the M\"obius randomness principle for additive twists. 
It is not known, even under the Generalized Riemann Hypothesis. 
The best known uniform bound, due to Baker and Harman \cite{BakerHarman1991}, is $\|F_R\|_{L^\infty} \ll_\epsilon R^{3/4 + \epsilon}$. 
Pandey and Radziwi{\l}{\l} \cite{PandeyRadziwill} note that improving this bound presents significant challenges, as any method based on Type-II sums can save at most $R^{1/4}$ over the trivial bound.
\end{remark}

\begin{theorem}\label{thm:conditional_improvement}
Under Assumption \ref{ass:l_infty},
\[
\FR(\mu_R) = R^{o(1)}.
\]
\end{theorem}

\begin{proof}
By Cauchy--Schwarz,
\[
\|F_R\|_{L^2}^2 = \int_0^1 |F_R(x)|^2 dx \le \|F_R\|_{L^\infty} \|F_R\|_{L^1}.
\]
Rearranging,
\[
\|F_R\|_{L^1} \ge \frac{\|F_R\|_{L^2}^2}{\|F_R\|_{L^\infty}}.
\]
Using $\|F_R\|_{L^2}^2 \sim \frac{6}{\pi^2} R$ and Assumption \ref{ass:l_infty},
\[
\|F_R\|_{L^1} \gg \frac{R}{R^{1/2+\epsilon}} = R^{1/2-\epsilon}.
\]
Hence
\[
\FR(\mu_R) = \frac{\|F_R\|_{L^1}}{\|F_R\|_{L^2}} \gg \frac{R^{1/2-\epsilon}}{R^{1/2}} = R^{-\epsilon}.
\]
Since $\epsilon > 0$ is arbitrary, $\FR(\mu_R) = R^{o(1)}$. 
\end{proof}

\begin{remark}
Theorem \ref{thm:conditional_improvement} shows that under Assumption \ref{ass:l_infty}, the Fourier Ratio decays slower than any power of $R$. 
This is a weaker statement than having a constant lower bound; it only guarantees subpolynomial decay. 
To obtain a constant lower bound, one would need a stronger assumption (e.g., $\|F_R\|_{L^\infty} \ll R^{1/2}$ with no $\epsilon$). 
We will not pursue this further here.
\end{remark}

\section{The Fourier Ratio and the Vapnik--Chervonenkis Dimension}
\label{sec:VC}

In this section we show that functions with Fourier Ratio bounded below by a positive constant form a class of high combinatorial complexity. 
We work on the discrete domain
\[
X_R = \{1,2,\dots,R\}.
\]
For a function $f: X_R \to \mathbb{C}$, define its Fourier transform by
\[
\hat{f}(x) = \sum_{n=1}^R f(n) e^{2\pi i n x}, \qquad x \in [0,1],
\]
and the Fourier Ratio by
\[
\FR(f) = \frac{\|\hat{f}\|_{L^1([0,1])}}{\|\hat{f}\|_{L^2([0,1])}}.
\]

Recall that $X_R' = \{ n \le R : \mu(n) \neq 0 \}$ is the set of squarefree numbers up to $R$, and that $|X_R'| \sim \frac{6}{\pi^2} R$. 
For any function $f: X_R' \to \{-1,1\}$, we extend $f$ to $X_R$ by setting $f(n)=0$ for $n \notin X_R'$; the Fourier transform and Fourier Ratio are then defined with respect to this extension.

For $\eta > 0$, define
\[
\mathcal{H}_R(\eta) = \left\{ f: X_R' \to \{-1,1\} : \FR(f) \ge \eta \right\}.
\]

\begin{definition}
A subset $S \subset X_R'$ is said to be \emph{shattered} by a class $\mathcal{F}$ of functions $f: X_R' \to \{-1,1\}$ if for every function $\sigma: S \to \{-1,1\}$ there exists $f \in \mathcal{F}$ such that
\[
f(n) = \sigma(n) \quad \text{for all } n \in S.
\]
The Vapnik--Chervonenkis dimension of $\mathcal{F}$, denoted $\VC(\mathcal{F})$, is the maximal cardinality of a shattered subset.
\end{definition}

\begin{theorem}\label{thm:vc_lower_proof}
There exists an absolute constant $c_0 > 0$ such that for all sufficiently large $R$, the class $\mathcal{H}_R(c_0)$ satisfies
\[
\VC(\mathcal{H}_R(c_0)) \ge c R
\]
for some absolute constant $c > 0$.
\end{theorem}

\begin{proof}
Since $|X_R'| \sim \frac{6}{\pi^2}R$, we may choose a subset $S \subset X_R'$ with
\[
|S| = \lfloor \alpha R \rfloor
\]
for some sufficiently small absolute constant $\alpha>0$. We will show that $S$ is shattered by $\mathcal H_R(c_0)$ for a suitable absolute constant $c_0>0$.

Fix an arbitrary function $\sigma:S\to\{-1,1\}$. We define a function $f:X_R'\to\{-1,1\}$ by setting
\[
f(n)=\sigma(n), \qquad n\in S,
\]
and choosing the values of $f$ on $X_R'\setminus S$ independently and uniformly at random from $\{-1,1\}$. Extend $f$ to $X_R$ by setting $f(n)=0$ for $n\notin X_R'$.

For $x\in[0,1]$, write
\[
\hat f(x)=A(x)+B(x),
\]
where
\[
A(x)=\sum_{n\in S}\sigma(n)e^{2\pi i nx},
\qquad
B(x)=\sum_{n\in X_R'\setminus S} f(n)e^{2\pi i nx}.
\]

For each fixed $x$, the random variable $B(x)$ is a Rademacher sum with coefficients of modulus $1$, so by Khintchine's inequality,
\[
\mathbb E|B(x)| \ge \frac{1}{\sqrt2}\sqrt{|X_R'\setminus S|}.
\]

On the other hand, since $\mathbb E B(x)=0$, Jensen's inequality gives
\[
\mathbb E|\hat f(x)|=\mathbb E|A(x)+B(x)| \ge \left|\mathbb E(A(x)+B(x))\right| = |A(x)|.
\]
Also, by the triangle inequality,
\[
|A(x)+B(x)| \ge |B(x)|-|A(x)|,
\]
and hence
\[
\mathbb E|\hat f(x)| \ge \mathbb E|B(x)|-|A(x)|.
\]

Therefore
\[
\mathbb E|\hat f(x)|
\ge
\max\left\{|A(x)|,\ \mathbb E|B(x)|-|A(x)|\right\}.
\]
For any nonnegative numbers $u,v$, one has
\[
\max\{u,v-u\}\ge \frac v2.
\]
Applying this with $u=|A(x)|$ and $v=\mathbb E|B(x)|$, we obtain
\[
\mathbb E|\hat f(x)| \ge \frac12\,\mathbb E|B(x)|
\ge
\frac{1}{2\sqrt2}\sqrt{|X_R'\setminus S|}.
\]

Since $|X_R'\setminus S| \sim \left(\frac{6}{\pi^2}-\alpha\right)R$, it follows that
\[
\mathbb E|\hat f(x)| \ge c_1\sqrt R
\]
for all $x\in[0,1]$ and some absolute constant $c_1>0$.

Integrating over $x\in[0,1]$,
\[
\mathbb E\|\hat f\|_{L^1}
=
\int_0^1 \mathbb E|\hat f(x)|\,dx
\ge
c_1\sqrt R.
\]
By averaging, there exists a realization of the random signs on $X_R'\setminus S$ such that
\[
\|\hat f\|_{L^1}\ge c_1\sqrt R.
\]

By Parseval's identity,
\[
\|\hat f\|_{L^2}^2
=
\sum_{n\in X_R'} |f(n)|^2
=
|X_R'|.
\]
Thus
\[
\|\hat f\|_{L^2}\le c_2\sqrt R
\]
for some absolute constant $c_2>0$ and all sufficiently large $R$.

Consequently, for this realization,
\[
\FR(f)=\frac{\|\hat f\|_{L^1}}{\|\hat f\|_{L^2}}
\ge
\frac{c_1}{c_2}
=:c_0>0.
\]
Since $f(n)=\sigma(n)$ on $S$, this shows that every sign pattern on $S$ is realized by some function in $\mathcal H_R(c_0)$. Therefore $S$ is shattered by $\mathcal H_R(c_0)$.

Because $|S|=\lfloor \alpha R\rfloor$, we conclude that
\[
\VC(\mathcal H_R(c_0))\ge \alpha R.
\]
Taking $c=\alpha$ completes the proof.
\end{proof}
\begin{corollary}\label{cor:vc_any_eta}
For any $\eta_R \le c_0$, we have $\VC(\mathcal{H}_R(\eta_R)) \ge \VC(\mathcal{H}_R(c_0)) \ge c R$ for the same constant $c > 0$.
\end{corollary}

\section{Proof of the Main Result}
\label{sec:learning}

We now combine the results of the previous sections to prove Theorem \ref{thm:main}. 
We recall the classical VC lower bound from learning theory.

\begin{theorem}[Vapnik--Chervonenkis lower bound \cite{VC1971, Ehrenfeucht1989}]
Let $\mathcal{H}$ be a hypothesis class of functions from a domain $X$ to $\{-1,1\}$ with Vapnik--Chervonenkis dimension $d \ge 1$. 
Then for any learning algorithm that receives $m$ labeled samples $(x_i, h^*(x_i))$ drawn independently from any distribution on $X$ and outputs a hypothesis $\hat{h}$, if the algorithm satisfies
\[
\sup_{h^* \in \mathcal{H}} \mathbb{P}\left( \operatorname{error}(\hat{h}) \ge \frac{1}{4} \right) \le \frac{1}{3}
\]
then necessarily $m \ge \frac{d}{32}$. 
Here $\operatorname{error}(h) = \mathbb{P}_{x \sim \mathcal{D}}(h(x) \neq h^*(x))$ for the true target function $h^*$.
\end{theorem}

\begin{proof}[Proof of Theorem \ref{thm:main}]
From Theorem \ref{thm:unconditional}, we have $\FR(\mu_R) \gg R^{-1/4-\epsilon}$. 
Choose a constant $c_1 > 0$ such that $\FR(\mu_R) \ge c_1 R^{-1/4-\epsilon}$ for all sufficiently large $R$. 
Let $\eta_R = c_1 R^{-1/4-\epsilon}$. 
For large $R$, $\eta_R \le c_0$, where $c_0$ is the constant from Theorem \ref{thm:vc_lower}. 
Hence $\mu_R \in \mathcal{H}_R(\eta_R)$.

By Corollary \ref{cor:vc_any_eta}, $\VC(\mathcal{H}_R(\eta_R)) \ge c R$ for some absolute constant $c > 0$. 
Therefore, by the VC lower bound theorem, any learning algorithm that succeeds uniformly on $\mathcal{H}_R(\eta_R)$ (and hence on $\mu_R$ as a member of this class) must use at least
\[
m \ge \frac{\VC(\mathcal{H}_R(\eta_R))}{32} \ge \frac{c}{32} R
\]
samples. 
Thus $m = \Omega(R)$, completing the proof.
\end{proof}

\begin{remark}
Although the lower bound $\FR(\mu_R)\gg R^{-1/4-\epsilon}$ is weaker than a constant lower bound, this does not affect the sample complexity conclusion. Indeed, for sufficiently large $R$, we have $\eta_R \le c_0$, and therefore
\[
\mathcal H_R(c_0) \subset \mathcal H_R(\eta_R).
\]
Since $\mathcal H_R(c_0)$ has Vapnik--Chervonenkis dimension at least $cR$, it follows that $\mathcal H_R(\eta_R)$ also has Vapnik--Chervonenkis dimension at least $cR$, because enlarging a hypothesis class cannot decrease its VC dimension. Consequently, the same lower bound $\Omega(R)$ on the number of samples holds.
\end{remark}

\section{Numerical Experiments}
\label{sec:numerics}

We supplement the theoretical results with numerical experiments for the Fourier Ratio
\[
\mathcal{R}(R)=\frac{\|F_R\|_{L^1([0,1])}}{\|F_R\|_{L^2([0,1])}},
\qquad
F_R(x)=\sum_{n=1}^R \mu(n)e^{2\pi i n x}.
\]

The purpose of this section is not to prove new results, but to illustrate the behavior of the Fourier Ratio for the M\"obius function and to compare it with other arithmetic functions. 
These computations are not used in the proofs. 
Their role is instead to provide intuition and to test whether the theoretical lower bounds are sharp.

Computations were done using SageMath.

\subsection{Benchmarks}

A natural reference value for $\mathcal{R}(R)$ comes from the following observation. 
If $Z$ is a centered real Gaussian random variable with variance $\sigma^2$, then
\[
\frac{\mathbb{E}|Z|}{(\mathbb{E}|Z|^2)^{1/2}} = \sqrt{\frac{2}{\pi}}.
\]
Since $\mathcal{R}(R) = \mathbb{E}_x |F_R(x)| / (\mathbb{E}_x |F_R(x)|^2)^{1/2}$ when $x$ is chosen uniformly from $[0,1]$, the constant $\sqrt{2/\pi} \approx 0.7979$ is the value one would obtain if the distribution of the values of $F_R(x)$ were Gaussian. 

We use this as a benchmark in the numerical plots. This behavior is consistent with rigorous results for random multiplicative functions: for Rademacher random multiplicative functions, exponential sums converge in distribution to Gaussian random variables, and in particular the Fourier Ratio remains bounded and converges to $\sqrt{2/\pi}$ (see \cite{RBN2020}).

\subsection{Results for the M\"obius function}

Table~\ref{tab:FR} records the main quantities for $R$ ranging from $10^2$ to $10^7$.

\begin{table}[ht]
\centering
\caption{The Fourier Ratio of the M\"obius function. Here $Q(R)/R$ is the squarefree density, and the reference values for the normalized $L^2$ and $L^1$ norms are $\sqrt{6/\pi^2}\approx 0.7797$ and $\sqrt{12/\pi^3}\approx 0.6220$, respectively.}
\label{tab:FR}
\begin{tabular}{r c c c c c}
\toprule
$R$ & $Q(R)/R$ & $\|F_R\|_{L^2}/\sqrt{R}$ & $\|F_R\|_{L^1}/\sqrt{R}$ & $\mathcal{R}(R)$ & $\|F_R\|_{L^\infty}/\sqrt{R}$ \\
\midrule
$10^2$             & 0.6100 & 0.7810 & 0.6160 & 0.7887 & 3.37 \\
$3 \times 10^2$    & 0.6100 & 0.7810 & 0.5890 & 0.7542 & 4.12 \\
$10^3$             & 0.6080 & 0.7797 & 0.6075 & 0.7790 & 4.09 \\
$3 \times 10^3$    & 0.6080 & 0.7797 & 0.6140 & 0.7875 & 4.56 \\
$10^4$             & 0.6083 & 0.7799 & 0.6202 & 0.7953 & 4.52 \\
$3 \times 10^4$    & 0.6081 & 0.7798 & 0.6228 & 0.7986 & 4.03 \\
$10^5$             & 0.6079 & 0.7797 & 0.6231 & 0.7992 & 3.89 \\
$3 \times 10^5$    & 0.6079 & 0.7797 & 0.6230 & 0.7991 & 3.87 \\
$10^6$             & 0.6079 & 0.7797 & 0.6226 & 0.7983 & 3.97 \\
$3 \times 10^6$    & 0.6079 & 0.7797 & 0.6225 & 0.7982 & 4.24 \\
$10^7$             & 0.6079 & 0.7797 & 0.6222 & 0.7980 & 4.61 \\
\bottomrule
\end{tabular}
\end{table}

Several features are visible in the table.

\smallskip\noindent\textbf{The $L^2$ normalization.}
The ratio $\|F_R\|_{L^2}/\sqrt{R}$ agrees very closely with the reference value $\sqrt{6/\pi^2} \approx 0.7797$. 
This is exactly what one expects from the classical squarefree asymptotic $Q(R)\sim \frac{6}{\pi^2}R$.

\smallskip\noindent\textbf{The $L^1$ normalization.}
The quantity $\|F_R\|_{L^1}/\sqrt{R}$ is numerically close to $\sqrt{12/\pi^3} \approx 0.6220$. 
At $R=10^7$, the observed value is $0.6222$.

\smallskip\noindent\textbf{The Fourier Ratio.}
The ratio $\mathcal{R}(R)$ approaches $\sqrt{2/\pi}$ rapidly after the smallest values of $R$. 
Over the range $10^4 \le R \le 10^7$, the computed values remain close to the Gaussian benchmark. 
This behavior is shown in Figure~\ref{fig:FR}.

\begin{figure}[ht]
\centering
\includegraphics[width=0.75\textwidth]{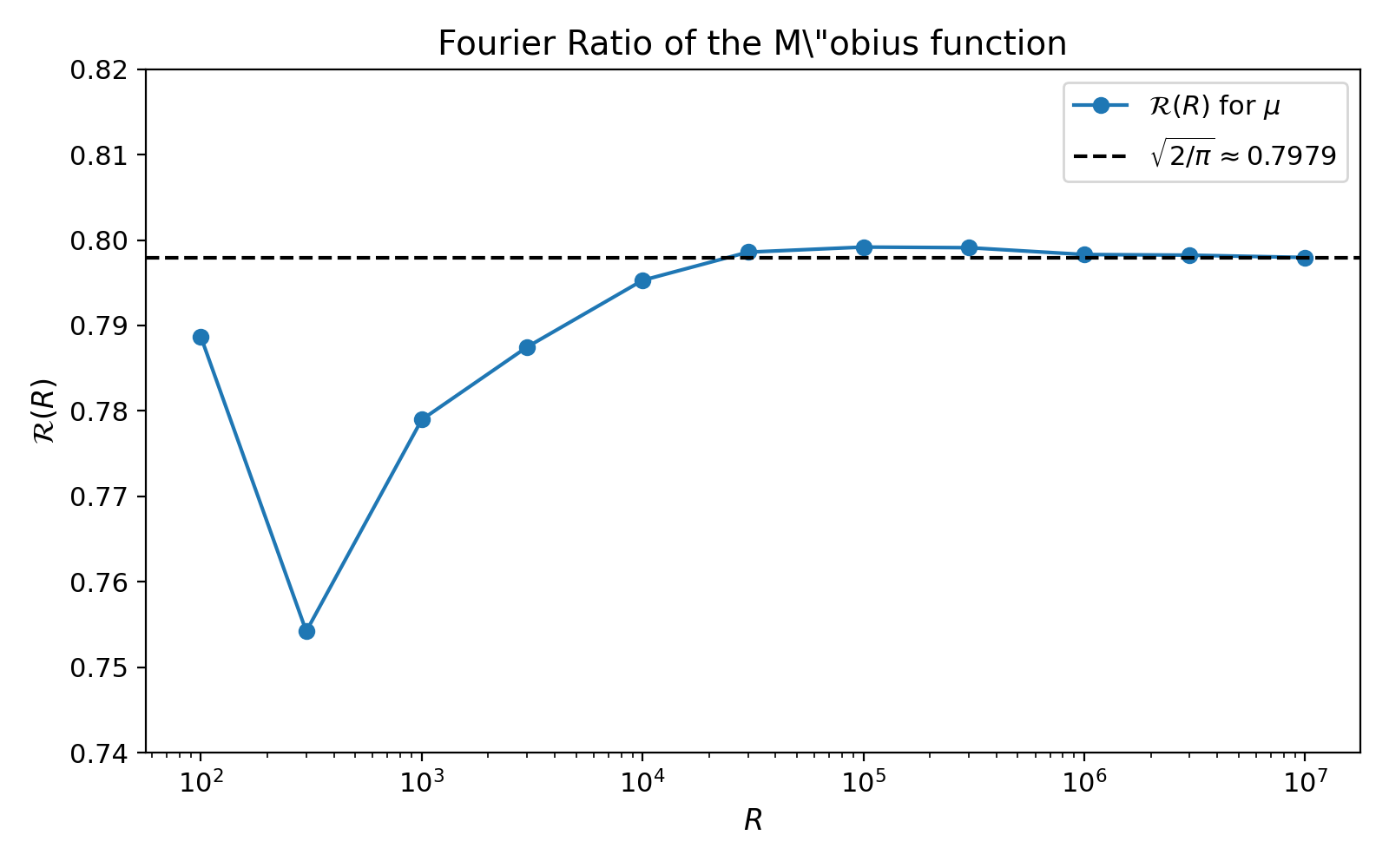}
\caption{The Fourier Ratio $\mathcal R(R)$ for the M\"obius function as a function of $R$ on a logarithmic scale. The dashed line marks the benchmark $\sqrt{2/\pi}\approx 0.7979$.}
\label{fig:FR}
\end{figure}

Figure~\ref{fig:convergence} plots the deviation $|\mathcal R(R)-\sqrt{2/\pi}|$ on a log--log scale.

\begin{figure}[ht]
\centering
\includegraphics[width=0.75\textwidth]{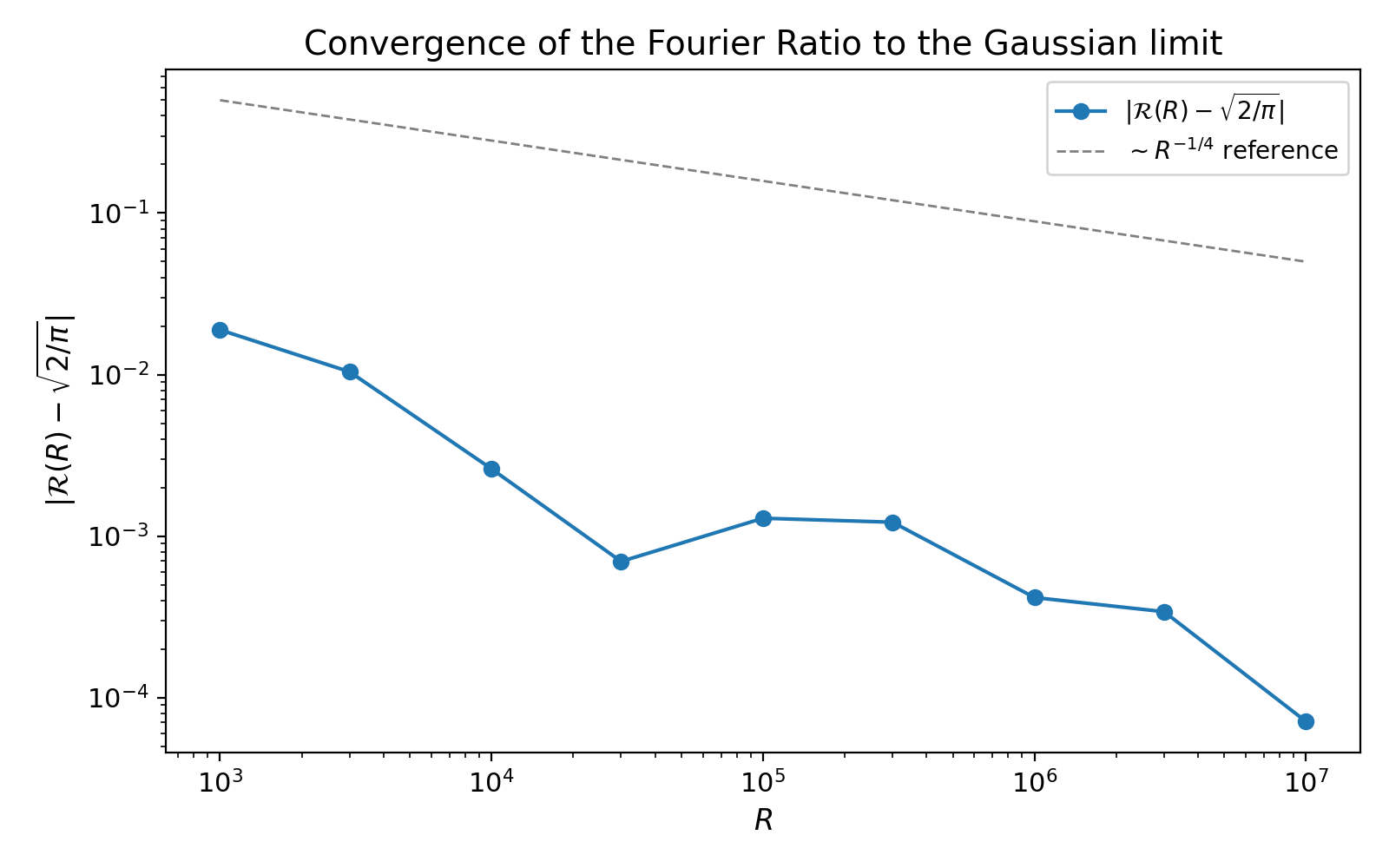}
\caption{The deviation $|\mathcal R(R)-\sqrt{2/\pi}|$ on a log--log scale for $R\ge 10^3$. The dashed line is a reference slope included only as a visual guide.}
\label{fig:convergence}
\end{figure}

\smallskip\noindent\textbf{Caveats.}
The numerical experiments are limited to $R \le 10^7$, which is very small by analytic number theory standards. 
Phenomena such as the Littlewood oscillation for $\pi(x)$ (where $\pi(x) - \operatorname{li}(x)$ changes sign infinitely often, with the first sign change occurring beyond $10^{14}$) serve as a cautionary tale: apparent convergence at small scales can be misleading. 
Thus the numerical evidence should be interpreted as suggestive rather than conclusive.

\subsection{Comparison with other arithmetic functions}

We compare the M\"obius function with several related examples at $R=10^6$. 
For functions with nonzero average, we consider centered versions to reflect oscillation rather than overall bias.

For each real-valued function $f$ on $\{1,\dots,R\}$, we form the exponential sum
\[
F_f(x)=\sum_{n=1}^R f(n)e^{2\pi i n x}, \qquad \mathcal{R}(f)=\frac{\|F_f\|_{L^1}}{\|F_f\|_{L^2}}.
\]
Table~\ref{tab:comparison} records the results.

\begin{table}[ht]
\centering
\caption{The Fourier Ratio for several arithmetic functions at $R=10^6$. The reference value for a centered real Gaussian is $\sqrt{2/\pi}\approx 0.7979$.}
\label{tab:comparison}
\begin{tabular}{l c c}
\toprule
Function $f(n)$ & $\mathcal{R}(f)$ & $\|F_f\|_{L^2}/\sqrt{R}$ \\
\midrule
$\mu(n)$ & 0.7983 & 0.7797 \\
$\lambda(n)$ & 0.7981 & 1.0000 \\
independent random signs $\varepsilon_n\in\{\pm1\}$ & 0.7981 & 1.0000 \\
$\Lambda(n)-c_{\Lambda}(R)$ & 0.6045 & 3.435 \\
$|\mu(n)|-c_{\mathrm{sf}}(R)$ & 0.1582 & 0.488 \\
\bottomrule
\end{tabular}
\end{table}

The M\"obius function, the Liouville function, and a sequence of independent random signs all yield Fourier Ratios very close to $\sqrt{2/\pi}$. 
By contrast, the centered von Mangoldt function and the centered squarefree indicator lie substantially below that level. 
Thus the Fourier Ratio separates these examples clearly at the scale tested. 
Figure~\ref{fig:comparison} illustrates this comparison.

\begin{figure}[ht]
\centering
\includegraphics[width=0.75\textwidth]{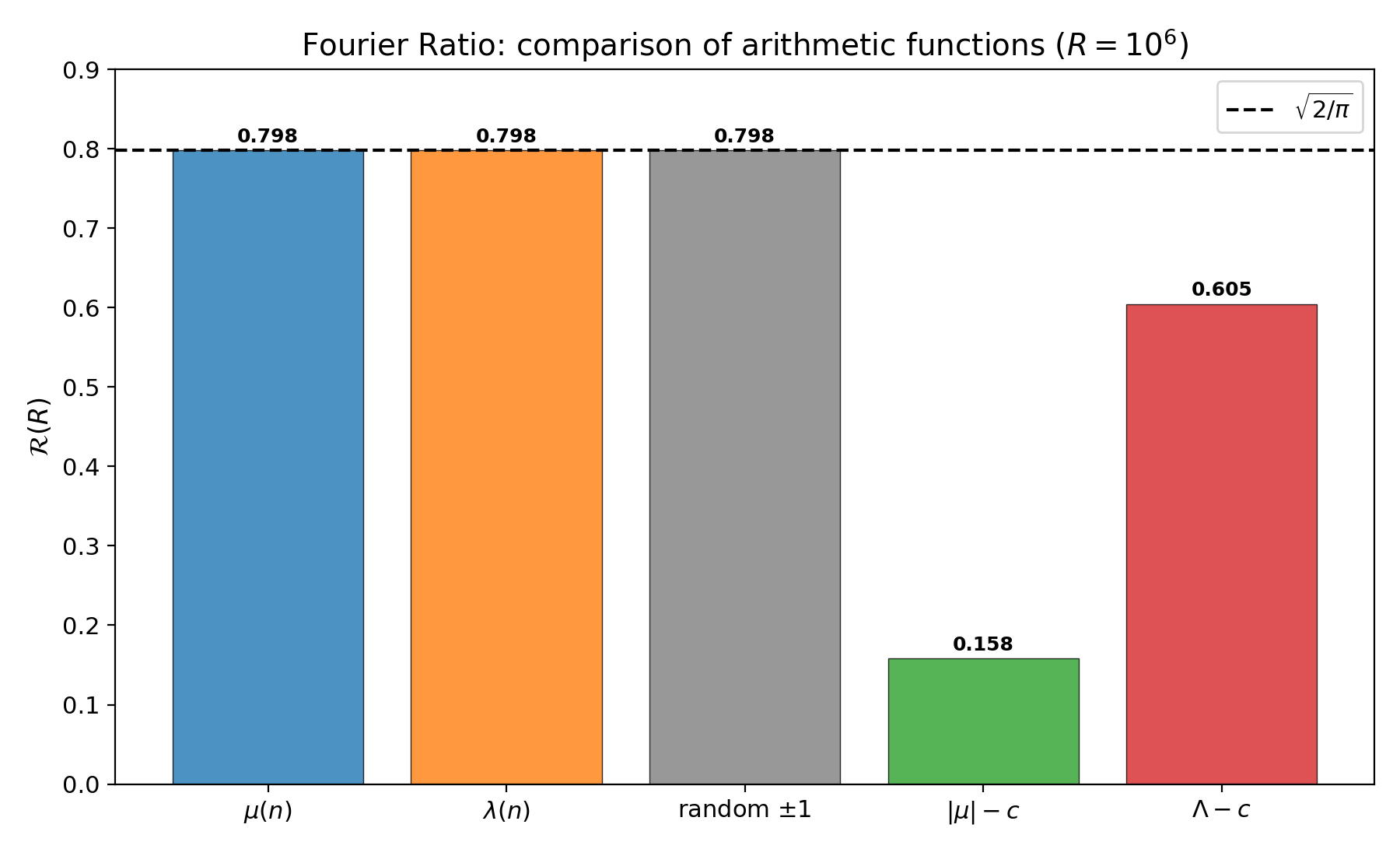}
\caption{The Fourier Ratio for five arithmetic functions at $R=10^6$. The dashed line marks the benchmark $\sqrt{2/\pi}$.}
\label{fig:comparison}
\end{figure}

\section{Open Problems}
\label{sec:problems}

The results of this paper raise several questions that we believe merit further investigation.

\begin{problem}[Improve the $L^1$ lower bound]
Pandey and Radziwi{\l}{\l} \cite{PandeyRadziwill} proved that $\|F_R\|_{L^1} \gg_\epsilon R^{1/4-\epsilon}$. 
Is the true order of growth $R^{1/2}$? 
This would be consistent with the Gaussian heuristic and with the conditional result under Assumption \ref{ass:l_infty}. 
Improving the exponent $1/4$ would require new insights into the distribution of zeros of Dirichlet $L$-functions.
\end{problem}

\begin{problem}[Determine the true growth of $\|F_R\|_{L^\infty}$]
The best known uniform bound, due to Baker and Harman \cite{BakerHarman1991}, is $\|F_R\|_{L^\infty} \ll_\epsilon R^{3/4+\epsilon}$. 
Is the true exponent $1/2$? 
This is a fundamental open problem in analytic number theory, closely related to the M\"obius randomness principle.
\end{problem}

\begin{problem}[Connection with the Riemann Hypothesis]
Assumption \ref{ass:l_infty} is much stronger than the Riemann Hypothesis. 
Is there any direct implication from the Riemann Hypothesis (or even the Generalized Riemann Hypothesis) to a non-trivial bound on $\|F_R\|_{L^\infty}$? 
Conversely, if the Riemann Hypothesis were false, would it force $\|F_R\|_{L^\infty}$ to be large? 
These questions remain open.
\end{problem}

\begin{problem}[VC dimension of $\mathcal{H}_R(\eta)$]
We have shown that for constant $\eta = c_0 > 0$, $\VC(\mathcal{H}_R(c_0)) \ge cR$. 
What is the VC dimension of $\mathcal{H}_R(\eta)$ for $\eta = \eta_R$ decaying to zero? 
Our argument shows that $\VC(\mathcal{H}_R(\eta_R)) \ge \VC(\mathcal{H}_R(c_0)) \ge cR$ whenever $\eta_R \le c_0$, but this does not preclude the possibility that the VC dimension is actually much larger for smaller $\eta_R$. 
Determining the precise dependence of VC dimension on the threshold $\eta$ is an interesting combinatorial problem.
\end{problem}

\begin{problem}[Other arithmetic functions]
The methods of this paper apply to any arithmetic function $f$ satisfying $\sum_{n=1}^R |f(n)|^2 \asymp R$ and the Pandey--Radziwill $L^1$ lower bound $\|F_f\|_{L^1} \gg R^{1/4-\epsilon}$. 
Which multiplicative functions satisfy such an $L^1$ lower bound? 
The results of \cite{PandeyRadziwill} suggest that a wide class of non-pretentious multiplicative functions may satisfy such bounds. 
Exploring the learnability of other arithmetic functions is a promising direction for future work.
\end{problem}

\begin{problem}[Algorithmic simplicity versus statistical complexity]
The M\"obius function is algorithmically simple: given $R$, one can compute $\mu(n)$ for all $n \le R$ by factoring each integer, so the Kolmogorov complexity of the sequence $(\mu(n))_{n\le R}$ is $O(\log R)$. 
Yet our results show that $\mu_R$ belongs to hypothesis classes that are statistically hard to learn in the distribution-independent setting, requiring $\Omega(R)$ samples for uniform success. This separation between algorithmic and statistical complexity is reminiscent of phenomena in computational learning theory (e.g., the existence of concept classes that are easy to represent but hard to learn). 
Can one construct explicit families of functions with maximal separation between Kolmogorov complexity and VC dimension? 
The M\"obius function provides a natural candidate from number theory.
\end{problem}

\end{document}